\documentclass{article}

\usepackage{fullpage}
\usepackage{amssymb}
\usepackage{amsmath}
\usepackage{amsthm}
\usepackage{mathrsfs}

\usepackage{fullpage}
\usepackage{physics2}
\usephysicsmodule{ab}
\usepackage{todonotes}
\usepackage{ifthen}
\usepackage{ stmaryrd }
\expandafter\def\csname opt@stmaryrd.sty\endcsname
{only,shortleftarrow,shortrightarrow}
\usepackage{extpfeil}

\usepackage{tikz-cd}
\usepackage{tikz}

\theoremstyle{definition}
\newtheorem{theorem}{Theorem}[section]

\newtheorem{lemma}[theorem]{Lemma}

\newtheorem{proposition}[theorem]{Proposition}
\newtheorem{definition}[theorem]{Definition}

\newtheorem{example}[theorem]{Example}

\newtheorem{remark}[theorem]{Remark}

\def\cC{{\cal C}}

\def\cZ{{\cal Z}}

\def\bfA{{\mathbf A}}
\def\bfC{{\mathbf C}}

\def\bfE{{\mathbf E}}
\def\bfF{{\mathbf F}}

\def\bfT{{\mathbf T}}

\def\bfn{{\mathbf n}}

\newcommand{\timesdots}{\times\dots\times}

\newcommand{\Hom}{\mathrm{Hom}}

\DeclareMathOperator{\Ext}{Ext}

\newcommand{\Mor}{\mathrm{Mor}}

\newcommand{\Alg}{\mathbf{Alg}}
\newcommand{\Id}{\mathrm{id}}
\newcommand{\Ob}{\mathrm{Ob}}
\newcommand{\Op}{\mathrm{op}}
\newcommand{\Der}{\mathrm{Der}}

\newcommand{\catab}{\mathbf{Ab}}
\newcommand{\catset}{\mathbf{Set}}

\newcommand{\catcat}{\mathbf{Cat}}

\newcommand{\catfam}{\mathbf{Fam}}
\newcommand{\catcfam}{\mathbf{CFam}}
\newcommand{\catnat}{\mathbf{Nat}}
\newcommand{\catcnat}{\mathbf{CNat}}
\newcommand{\catccnat}{\mathbf{CCNat}}
\newcommand{\catccc}{\mathbf{CCC}}
\newcommand{\BW}{\mathrm{BW}}

\newcommand{\cattheories}{\mathbf{Law}}
\newcommand{\catfact}[1]{\mathcal{F}#1}

\newcommand{\Ty}{\mathsf{BiMag}}
\newcommand{\ev}[2]{\mathrm{ev}^{#1}_{#2}}

\newcommand{\HQ}{H\!Q}

\newcommand{\varlaw}{\mathbf{L}}
\newcommand{\varccc}{\mathbf{C}}
\newcommand{\varsorts}{S}

\newcommand{\varobj}[1][0]{%
  {\ifthenelse{\equal{#1}{0}}{X}{%
    \ifthenelse{\equal{#1}{1}}{Y}{Z}%
  }}%
}

\newcommand{\varcat}{\mathbf{C}}

\newcommand{\varsign}{\Sigma}
\newcommand{\varsortedmor}{\iota}

\newcommand{\abobjs}{\mathrm{ab}}
\newcommand{\currying}[1]{\lambda #1}

\newcommand{\fstspec}{{}'\!E}
\newcommand{\sndspec}{{}''\!E}

\newcommand{\projection}{\pi}

\makeatletter
\newcommand*{\relrelbarsep}{.386ex}
\newcommand*{\relrelbar}{%
  \mathrel{%
    \mathpalette\@relrelbar\relrelbarsep
  }%
}
\newcommand*{\@relrelbar}[2]{%
  \raise#2\hbox to 0pt{$\m@th#1\relbar$\hss}%
  \lower#2\hbox{$\m@th#1\relbar$}%
}
\providecommand*{\rightrightarrowsfill@}{%
  \arrowfill@\relrelbar\relrelbar\rightrightarrows
}
\providecommand*{\leftleftarrowsfill@}{%
  \arrowfill@\leftleftarrows\relrelbar\relrelbar
}
\providecommand*{\xrightrightarrows}[2][]{%
  \ext@arrow 0359\rightrightarrowsfill@{#1}{#2}%
}
\providecommand*{\xleftleftarrows}[2][]{%
  \ext@arrow 3095\leftleftarrowsfill@{#1}{#2}%
}
\makeatother

\title{Cohomology of Small Cartesian Closed Categories}
\author{Mirai Ikebuchi}
\begin{document}

\maketitle

\begin{abstract}
We show the isomorphism between the Quillen cohomology and the Baues-Wirsching cohomology of a cartesian closed category (CCC).
This is an extension of the results of Dwyer-Kan for small categories and Jibladze-Pirashvili for small categories with finite products.
These results implies that The Quillen cohomology of a CCC $\bfC$ coincides with that of $\bfC$ as a category with finite products, and also that of $\bfC$ as a small category.
\end{abstract}

\section{Introduction}
In this paper, we study cohomology of small cartesian closed categories (CCCs).
One of the motivations to study cohomology of categories with some structures is mathematical logic:
In categorical logic, it is known that there are correspondences between categories and formal theories, which mean a set of axioms.
Small categories with finite products and (first-order) equational theories \cite{lawvere63} and between small CCCs and higher-order equational theories, a set of equations between $\lambda$-terms in typed $\lambda$-calculus \cite{lambek88,crole93}.
Therefore, (co)homology of such classes of categories is an invariant of such theories, and some applications of them have been studied \cite{mm16,ikebuchi22, ikebuchi21, i25}.

We focus on two cohomology theories here: Baues--Wirsching cohomology and Quillen (or Andr\'e--Quillen) cohomology.
Baues--Wirsching cohomology \cite{baues85} is a cohomology of a small category $\bfC$ with coefficients in a \emph{natural system} over $\bfC$.
This cohomology is a generalization of Hochschild--Mitchell cohomology with coefficients in a $\bfC$-bimodule \cite{m72} and the cohomology of the classifying space of $\bfC$ with coefficients in a local system.

Quillen cohomology \cite{quillen2006homotopical} is defined for objects in algebraic categories using Quillen's homotopical algebra.
A coefficient module for the Quillen cohomology of $C$ in an algebraic category $\cC$ is an abelian group object in the overcategory $\cC/C$, called a \emph{Beck module} over $C$.

For a set $O$, let $\catcat_O$ be the category of small categories whose object set of objects is $O$ and whose morphisms are functors that map objects identically.
Then, it is known that the category of natural systems over $\bfC$ is equivalent to the category $(\catcat_O/\bfC)_\abobjs$ of Beck modules over $\bfC\in \Ob(\catcat_O)$, and
that there is an isomorphism between Quillen cohomology and Baues--Wirsching cohomology \cite{dk88},
\begin{equation}\label{eqn:catiso}
\HQ^n_{\catcat_O}(\bfC;D)\cong H_\BW^{n+1}(\bfC;D).
\end{equation}

In \cite{jp06}, Jibladze and Pirashvili studied cohomology of small categories with finite products, called \emph{Lawvere theories}.
They first defined the notion of \emph{cartesian natural systems} over a Lawvere theory and showed that the category $\catcnat_\bfC$ of cartesian natural systems over a Lawvere theory $\bfC$ is equivalent to $(\cattheories_\varsorts/\bfC)_\abobjs$ where $\cattheories_\varsorts$ for a set $\varsorts$ is the category of $\varsorts$-\emph{sorted} Lawvere theories.
Here, an $\varsorts$-sorted Lawvere theory is a Lawvere theory such that its objects are the elements of $\varsorts$ and their formal finite products, and a projection $\pi_i : \prod_{j=1}^n X_j \to X_i$ is specified for each finite family $\{X_j\}_{j=1,\dots,n}$ and $i=1,\dots,n$.
Morphisms between $\varsorts$-sorted Lawvere theories are functors that map objects identically and preserve the specified projections.
Then, they showed
\[
\HQ^n_{\cattheories_\varsorts}(\bfC;D) \cong H_\BW^{n+1}(\bfC;D)
\]
for any $n > 0$ and any cartesian natural system $D$ over an $\varsorts$-sorted Lawvere theory $\bfC$.
This result is interesting not only on its own, but also because, combining with \eqref{eqn:catiso}, it yields the isomorphism
\[
\HQ^n_{\cattheories_\varsorts}(\bfC;D) \cong \HQ^n_{\catcat_O}(\bfC;D)
\]
for $n>0$, a cartesian natural system $D$ over an $\varsorts$-sorted Lawvere theory $\bfC$, and $O = \Ob(\bfC)$.

The main result of this paper is that these isomorphisms extend to (small) CCCs.
More precisely, we
\begin{itemize}
	\item define the notion of $\varsorts$-\emph{sorted} CCCs for a set $\varsorts$ and show that the category $\catccc_\varsorts$ of $\varsorts$-sorted CCCs is algebraic (Subsection 3.2),
	\item define the notion of \emph{cartesian closed natural system} over a CCC $\bfC$ and show an equivalence $\catccnat_\bfC \simeq (\catccc_\varsorts/\bfC)_\abobjs$ between the category of cartesian natural systems over $\bfC$ is equivalent to the category of Beck modules over $\bfC$ as an $\varsorts$-sorted CCC (Subsection 5.3),
	\item and show that, for a Beck module $D$,
\[
\HQ^n_{\catccc_\varsorts}(\bfC;D)  \cong H_\BW^{n+1}(\bfC;D)
\] for positive $n$ (Section 7).
\end{itemize}

Because any CCC $\bfC$ can be thought of as an $\varsorts'$-sorted Lawvere theory and cartesian closed natural systems are cartesian natural systems by definition, we have
\begin{align*}
&\HQ^n_{\catccc_\varsorts}(\bfC;D) \cong \HQ^n_{\cattheories_{\varsorts'}}(\bfC; D)\cong \HQ^n_{\catcat_O}(\bfC;D) \cong H_\BW^{n+1}(\bfC;D).
\end{align*}

{\bf Acknowledgments:} The author would like to thank Haynes Miller for fruitful discussions and comments.

\section{Universal algebra}
In this section, we recall some notions in universal algebra. For more details, see \cite{cohn2012universal} for example.

\begin{definition}
Let $\varsorts$ be a set of sorts and $V_X$ be an infinite set of \emph{variable symbols of sort $X$} for each sort $X \in \varsorts$.
Let $V = \coprod_{X \in \varsorts} V_X$.
\begin{itemize}
	\item An $\varsorts$-\emph{sorted signature} is a set $\varsign$ (of \emph{operation symbols}) together with a function $\alpha : \varsign \to \varsorts^* \times \varsorts$.
If $\alpha(f) = (X_1\dots X_n,X)$ for $f \in \varsign$, we write $f : X_1\timesdots X_n \to X$.
	\item A \emph{term of sort} $X$ over $\varsign$ and $V$ is defined inductively as follows. (i) Any variable symbol $x \in V$ of sort $X$ is a term of sort $X$. (ii) If $f : X_1\timesdots X_n \to X$ and $t_1,\dots,t_n$ are terms of sorts $X_1,\dots,X_n$, respectively, then the formal expression $f(t_1,\dots,t_n)$ is a term of sort $X$.
	\item $T_\varsign(V)$ denotes the set of all terms over $\varsign$ and $V$.
	\item a finite list of variables $x_1\dots x_n\in V^*$ is called a \emph{context}. We often write a context as $x_1:X_1,\dots,x_n:X_n$ where $X_i$ is a sort of $x_i$.
	\item For a term $t$ and a context $x_1:X_1,\dots,x_n:X_n$ such that any variable in $t$ is in $\{x_1,\dots,x_n\}$, we call the formal expression $x_1:X_1,\dots,x_n:X_n \vdash t$ a \emph{term-in-context}.
	\item An \emph{equation} is a pair $(t_1,t_2)$ of two terms in $T_\varsign(V)$. An equation is written as $t_1 \approx t_2$.
	\item A pair $(\varsign, E)$ of a signature and a set of equations is called an \emph{($\varsorts$-sorted) equational presentation}.
\end{itemize}
\end{definition}
\begin{example}[Abelian groups]\label{ex:abelian}
Let $\varsorts$ be a singleton set $\{X\}$ and $\varsign$ be $\{0,-,+\}$ with $\alpha(0) = (\epsilon, X)$, $\alpha(-) = (X,X)$, $\alpha(+) = (XX, X)$.
We write $t_1+t_2$ for the term $+(t_1,t_2)$.
The following set $E$ of equations presents the theory of abelian groups:
\[
x \cdot 0 \approx x,\quad x+ (-x) \approx 0, \quad (x_1+ x_2)+ x_2 \approx x_1+ (x_2+ x_3), \quad x_1 + x_2 \approx x_2 + x_1
\]
where $x_1,x_2,x_3 \in V = V_X$.
\end{example}
\begin{example}[Left modules over a monoid]\label{ex:module}
Let $\varsorts' = \{X,Y\}$ and $\varsign' = \{0,-,+,1,\circ,\cdot\}$ with $\alpha'(1) = (\epsilon,Y)$ $\alpha'(\circ) = (YY,Y)$, $\alpha'(\cdot) = (YX,X)$, and $\alpha'$ is defined in the same way as in the previous example for $0,-,+$.
We write $t_1\circ t_2$ for $\circ(t_1,t_2)$ and $t_1\cdot t_2$ for the term $\cdot(t_1,t_2)$.
Then, the following set $E'$ of equations together with the equations in the previous example presents the theory of left modules over a monoid:
\begin{align*}
&1 \circ y \approx y,\quad y \circ 1 \approx y,\quad (y_1 \circ y_2)\circ y_3 \approx y_1 \circ (y_1 \circ y_3),\\
&y \cdot 0 \approx 0,\quad 1 \cdot x \approx x,\quad (y_1 \circ y_2)\cdot x \approx y_1\cdot (y_2\cdot x), \quad y\cdot (x_1 + x_2) \approx y\cdot x_1 + y\cdot x_2
\end{align*}
where $x,x_i \in V_X$ and $y,y_i \in V_Y$.
That is, the equations in the first line say that $\circ$ is the monoid multiplication with the unit $1$, and the equations in the second line are the laws for the scalar multiplication $y \cdot x$.
\end{example}
\begin{definition}
Fix a set of sorts $\varsorts$, a set of variables $V = \coprod_{X \in \varsorts} V_X$, an $\varsorts$-sorted signature $\varsign$.
\begin{itemize}
	\item A $\varsign$-\emph{algebra} is a family $A = \{A_X\}_{X \in \varsorts}$ of sets $A_X$ ($X \in \varsorts$) together with a function $\llbracket f \rrbracket : A_{X_1}\timesdots A_{X_n} \to A_X$ for each operation symbol $f : X_1\timesdots X_n \to X$.
	\item For two $\varsign$-algebras $A,A'$, a morphism between them consists of functions $\phi_X : A_X \to A'_X$ for each $X \in \varsorts$ that commutes with $\llbracket f\rrbracket$ for all $f \in \varsign$.
	\item Let $A$ be a $\varsign$-algebra. Given a term $t \in T_\varsign(V)$ of sort $X$ and a family of functions $v = \{v_X : V_X \to A_X\}_{X \in \varsorts}$, the \emph{interpretation} $\llbracket t \rrbracket_{v}$ is the element of $A_X$ defined inductively as follows. (i) If $t = x_i$ of sort $X_i$, then $\llbracket t \rrbracket_{v} = v_{X_i}(x_i)$. (ii) If $t = f(t_1,\dots,t_n)$, then $\llbracket t \rrbracket_{v} = \llbracket f \rrbracket(\llbracket t_1 \rrbracket_{v},\dots,\llbracket t_n \rrbracket_{v})$.
	\item A $\varsign$-algebra $A$ is said to \emph{satisfy} an equation $t_1 \approx t_2$ if $\llbracket t_1\rrbracket_{v} = \llbracket t_2\rrbracket_{v}$ holds for any $v$.
	\item For a set of equations $E$, the $\varsign$-algebras satisfying all equations in $E$ and morphisms between them form a category denoted by $\Alg(\varsign,E)$.
\end{itemize}
\end{definition}

It is not difficult to show that, for $\varsign$, $E$ defined in Example \ref{ex:abelian}, the category $\Alg(\varsign,E)$ is equivalent to the category of abelian groups $\catab$.
For $\varsign'$, $E'$ defined in Example \ref{ex:module}, the category $\Alg(\varsign',E')$ is equivalent to the category whose objects are pairs $(M,A)$ of a monoid $M$ and a left module $A$ over $M$, and morphisms from $(M,A)$ to $(M',A')$ are pairs $(\phi_M,\phi_A)$ of a monoid homomorphism $\phi_M : M \to M'$ and a group homomorphism $\phi_A : A \to A'$ such that
$\phi_M(y) \cdot \phi_A(x) = \phi_A(y\cdot x)$.

We define the notion that an equation $t \approx s$ can be proved from a set $E$ of equations, written $E \vdash t \approx s$, as follows.
\begin{itemize}
	\item For any $t \approx s \in E$, $E \vdash t \approx s$.
	\item For any term $t$, $E \vdash t\approx t$.
	\item If $E \vdash t \approx s$, then $E \vdash s \approx t$.
	\item If $E \vdash t \approx s$ and $E \vdash s \approx u$, then $E \vdash t \approx u$.
	\item Let $t$, $s$ be terms, $x_1,\dots,x_n$ be the variables that occur in $t$ or $s$ whose sorts are $X_1,\dots,X_n$, and $t_1,\dots,t_n$ be terms of sorts $X_1,\dots,X_n$, respectively.
If $E\vdash t\approx s$, then $E \vdash t[t_1/x_1,\dots,t_n/x_n] \approx s[t_1/x_1,\dots,t_n/x_n]$.
Here, $t[t_1/x_1,\dots,t_n/x_n]$ (resp. $s[t_1/x_1,\dots,t_n/x_n]$) is the term obtained from $t$ (resp. $s$) by replacing each variable $x_i$ with $t_i$.
	\item For any $f \in \varsign$ with $\alpha(f) = (X_1\dots X_n, X)$ and terms $t_1,\dots,t_n$, $s_1,\dots,s_n$ such that $t_i$ and $s_i$ are of sort $X_i$ ($i=1,\dots,n$), if $E \vdash t_i \approx s_i$ for each $i=1,\dots,n$, then $E \vdash f(t_1,\dots,t_n) \approx f(s_1,\dots,s_n)$.
\end{itemize}

We say that $t\approx s$ is a \emph{semantic consequence} of $E$, written $E \models t \approx s$, for any $\varsign$-algebra $A$ satisfying all equations in $E$, $A$ satisfies $t\approx s$.
The following ensures that the two relations $E \vdash t \approx s$ and $E \models t \approx s$ coincide.

\begin{theorem}
$E \vdash t \approx s$ if and only if $E \models t \approx s$.
\end{theorem}

\section{Lawvere theories and cartesian closed categories}
In this section, we define the notion of $\varsorts$-sorted cartesian closed categories and show that the category of them is algebraic.

\subsection{Lawvere theories}

\begin{definition}
Let $\varlaw$ be a category with finite products.
A \emph{model} of $\varlaw$ is a product-preserving functor $\varlaw \to \catset$, and a \emph{morphism of models} is a natural transformation.
We write $\Alg\varlaw$ for the category of models of $\varlaw$.
\end{definition}

\begin{definition}
If a category $\varcat$ is equivalent to $\Alg\varlaw$ for some Lawvere theory $\varlaw$, we say that $\varcat$ is \emph{algebraic}.
\end{definition}

Let $\bfn$ be the set $\ab\{1,\dots,n\}$ for $n=0,1,\dots$.
For a set $\varsorts$, let $\catfam_\varsorts^\Op$ be the full subcategory of $\catset/\varsorts$ with objects $f: \bfn \to \varsorts$ for $n=0,1,\dots$.
Note that $\catfam_\varsorts^\Op$ has finite coproducts $f_1 + f_2 : \bfn_1 + \bfn_2 \to \varsorts$ for $f_i : \bfn_i \to \varsorts$ ($i=1,2$), so $\catfam_\varsorts = (\catfam_\varsorts^\Op)^\Op$ has finite products.
We write $\varsorts^*$ for $\Ob(\catfam_\varsorts)$.

Note that any object $f : \bfn \to S$ can be thought of as a string $X_1\dots X_n$ over $S$ for $f(i) = X_i$, and the product of $X_1\dots X_{n_1}$ and $Y_1\dots Y_{n_2}$ is the concatenation $X_1\dots X_{n_1} Y_1\dots Y_{n_2}$.
A morphism $X_1\dots X_n \to Y_1\dots Y_m$ is a function $u : \mathbf{m} \to \mathbf{n}$ such that $X_{u(i)} = Y_i$ for each $i\in \mathbf{m}$.
Also, we can check that $X_1\dots X_n$ is isomorphic to any permutation $X_{\sigma(1)}\dots X_{\sigma(n)}$ where $\sigma : \bfn \to \bfn$ is a bijection.

\begin{definition}
\begin{itemize}
	\item For a set $\varsorts$, an $\varsorts$-\emph{sorted Lawvere theory} is a small category $\varlaw$ that has finite products together with a morphism $\iota : \catfam_\varsorts \to \varlaw$ of Lawvere theories such that the objects of $\varlaw$ are functions $\bfn \to \varsorts$ and $\iota$ is identity on objects.
We often call $\varlaw$ an $\varsorts$-sorted Lawvere theory without mentioning $\iota$.
	\item A morphism between $\varsorts$-sorted Lawvere theories $\iota : \catfam_\varsorts \to \varlaw$, $\iota' : \catfam_\varsorts \to \varlaw'$ is a functor $F : \varlaw \to \varlaw'$ that identity on objects, preserves products, and satisfy $F \circ \iota = \iota'$.
	\item The category of $\varsorts$-sorted Lawvere theories is denoted by $\cattheories_\varsorts$.
\end{itemize}
\end{definition}

In other words, a category $\varlaw$ is an $\varsorts$-sorted Lawvere theory if $\Ob(\varlaw) = \Ob(\catfam_\varsorts)$ and if each projection $\varobj_1\dots \varobj_k \to \varobj_i$ is specified for any objects $\varobj_1,\dots,\varobj_k$.
A functor between $\varsorts$-sorted Lawvere theories $\varlaw \to \varlaw'$ is a morphism of $\varsorts$-sorted Lawvere theories if it is a morphism of Lawvere theories and preserves the specified projections.

The following ensures that, modulo equivalence, $\varsorts$-sorted Lawvere theories are not actually a special case of small category with finite products.
\begin{proposition}\cite{jp06}\label{prop:sortedtheory}
For any small category $\varlaw$ with finite products, there exists a set $\varsorts$ and an $\varsorts$-sorted Lawvere theory $\catfam_\varsorts \to \varlaw^*$ such that $\varlaw$ is equivalent to $\varlaw^*$.
\end{proposition}
\begin{proof}
Take $\varsorts = \Ob(\varlaw)$, $\Ob(\varlaw^*) = \varsorts^*$, and
\[
\Hom_{\varlaw^*}(X_1\dots X_n,Y_1 \dots Y_m) = \Hom_\varlaw\ab(X_1\times \dots \times X_n, Y_1\times \dots \times Y_m).
\]
Since we have an evident functor $\varlaw^* \to \varlaw$ that is full and faithful and surjective on objects, $\varlaw^*$ is equivalent to $\varlaw$.
Define an identity-on-objects functor $F : \catfam_\varsorts \to \varlaw^*$  as $Fu = \ab<\pi_{u(1)},\dots,\pi_{u(m)}>$ for $u : X_1\dots X_n \to Y_1\dots Y_m$ in $\catfam_\varsorts$ where $\pi_{u(i)}$ is the projection $X_1\dots X_n \to X_{u(i)} = Y_i$ in $\varlaw^*$.
Therefore $\varlaw^*$ forms an $\varsorts$-sorted Lawvere theory.
\end{proof}

We have the \emph{forgetful} functor $U_{\cattheories_\varsorts}: \cattheories_\varsorts \to \catset^{\varsorts^* \times \varsorts}$ where $\catset^{\varsorts^*\times \varsorts}$ is the functor category from the discrete category $\varsorts^*\times \varsorts$ to $\catset$ defined as
\begin{align*}
U_{\cattheories_\varsorts}(\iota:\catfam_\varsorts \to \varlaw)(X,Y) &= \Hom_{\varlaw}(X,Y)\\
U_{\cattheories_\varsorts}(f : \varlaw \to \varlaw')_{(X,Y)} &= f|_{\Hom_\varlaw(X,Y)} : \Hom_{\varlaw}(X,Y) \to \Hom_{\varlaw'}(X,Y).
\end{align*}

\begin{proposition}\cite{lawvere63}
$U_{\cattheories_\varsorts}$ has a left adjoint $F_{\cattheories_\varsorts}: \catset^{\varsorts^*\times \varsorts} \to \cattheories_\varsorts$, called the \emph{free functor}.
\end{proposition}

Any Lawvere theory has an \emph{equational presentation} defined as follows.

\begin{proposition}\cite[Proposition 14.28]{Adamek2010}\label{prop:lawvere-eq}
For any  $\varsorts$-sorted equational presentation $(\varsign,E)$, there is an $\varsorts$-sorted Lawvere theory $\varlaw$ such that $\Alg\varlaw \cong \Alg(\varsign, E)$ and vice versa.
\end{proposition}
\begin{proof}(Sketch)
Let $(\Sigma,E)$ be an $\varsorts$-sorted equational presentation.
We construct a Lawvere theory $\varlaw$ as follows.
First, the objects of $\varlaw$ are the contexts $x_1:X_1,\dots,x_n:X_n$.
A morphism from $x_1:X_1,\dots,x_n:X_n$ to $y_1:Y_1,\dots,y_m:Y_m$ is an equivalence class of an $m$-tuple of terms in context $x_1:X_1,\dots,x_n:X_n \vdash (t_1,\dots,t_m)$ where $x_1:X_1,\dots,x_n:X_n \vdash (t_1,\dots,t_m)$ and $x_1:X_1,\dots,x_n:X_n \vdash (s_1,\dots,s_m)$ are equivalent if $E \vdash t_i \approx s_i$ for each $i=1,\dots,m$.
We write the equivalence class of $x_1:X_1,\dots,x_n:X_n \vdash (t_1,\dots,t_m)$ as $x_1:X_1,\dots,x_n:X_n \mid (t_1,\dots,t_m)$.

The composition $(y_1:Y_1,\dots,y_m:Y_m \mid (s_1,\dots,s_k))\circ (x_1:X_1,\dots,x_n:X_n \mid (t_1,\dots,t_m))$ is defined as
\[
x_1:X_1,\dots,x_n:X_n \mid (s_1[t_1/y_1,\dots,t_m/y_m],\dots,s_k[t_1/y_1,\dots,t_m/y_m]).
\]
The identity morphism on $x_1:X_1,\dots,x_n:X_n$ is $x_1:X_1,\dots,x_n:X_n \mid (x_1,\dots,x_n)$.
We can check that the product of $x_1:X_1,\dots,x_n:X_n$ and $y_1:Y_1,\dots,y_m:Y_m$ is $x_1:X_1,\dots,x_n:X_n,y_1:Y_1,\dots,y_m:Y_m$ together with projections $x_1:X_1,\dots,x_n:X_n,y_1:Y_1,\dots,y_m:Y_m \mid (x_1,\dots,x_n)$ and $x_1:X_1,\dots,x_n:X_n,y_1:Y_1,\dots,y_m:Y_m \mid (y_1,\dots,y_m)$. The terminal object is the empty context and the unique morphism from a context to the empty context is $x_1:X_1,\dots,x_n:X_n \mid ()$.

We say $\varlaw$ the $\varsorts$-sorted Lawvere theory \emph{presented} by $(\varsign,E)$.
\end{proof}

\begin{proposition}
For a set $O$, let $\catcat_O$ be the category such that $\Ob(\catcat_O)$ consists of small categories whose object set is $O$ and $\Mor(\catcat_O)$ consists of functors that are identity on objects.
Then, $\catcat_O$ is algebraic.
\end{proposition}
\begin{proof}
We construct a set of sorts and an equational presentation $(\varsign^{\catcat_O},E^{\catcat_O})$ as follows.
\begin{itemize}
	\item A sort is a pair $(X, Y)$ of $X,Y \in O$. That is, our set of sorts is $O\times O$.
	\item $\varsign^{\catcat_O}$ consists of operation symbols
\[
{-}\circ_{X,Y,Z}{-} : (Y,Z)\times (X,Y) \to (X,Z) \quad \text{and}\quad \Id_{X} : 1 \to (X, X)
\]
for each $X,Y,Z \in O$.
We often ommit the subscripts and superscripts ${X},{Y},{Z}$ and just write $\circ$, $\Id$.
	\item $E^{\catcat_O} = E^{\catcat_O}_\text{comp} \cup E^{\catcat_O}_\text{id}\cup E^{\catcat_O}_\text{pair}$ where
\begin{itemize}
	\item $E^{\catcat_O}_\text{comp}$ consists of $(x\circ_{X,Y,Z} y) \circ_{W,X,Z} z \approx x\circ_{W,Y,Z} (y\circ_{W,X,Y} z)$ for each $W,X,Y,Z\in O$,
	\item $E^{\catcat_O}_\text{id}$ consists of $x\circ_{X,X,Y} \Id_X \approx x$, $\Id_Y\circ_{X,Y,Y} x \approx x$ for each $X,Y\in O$,
\end{itemize}
\end{itemize}
Then, any algebra $\ab\{A_{(X, Y)}\}$ of $(\Sigma^{\catcat_O},E^{\catcat_O})$ can be seen as a category $\bfC$ such that $\Ob(\bfC)=O$, $\Hom_\bfC(X,Y) = A_{(X,Y)}$, and the composition and identities are given by $\llbracket \circ\rrbracket$, $\llbracket \Id\rrbracket$.
\end{proof}

\begin{proposition}\label{prop:lawvere-alg}
$\cattheories_\varsorts$ is algebraic.
\end{proposition}
\begin{proof}
Our set of sorts is $\varsorts^*\times \varsorts^*$ and define an equational presentation $(\Sigma^{\cattheories_\varsorts},E^{\cattheories_\varsorts})$ as $\Sigma^{\cattheories_\varsorts} = \Sigma^{\catcat_{S^*}} \cup \Sigma^{\cattheories_\varsorts}_{\mathrm{pair}}$ and $E^{\cattheories_\varsorts} = E^{\catcat_{S^*}} \cup E^{\cattheories_{S}}_{\mathrm{pair}}$ where $\Sigma^{\cattheories_\varsorts}_{\mathrm{pair}}$ consists of
\[
\ab<-,\dots,->_{X,Y_1,\dots,Y_m} : (X,Y_1)\timesdots (X,Y_m) \to (X,Y_1\timesdots Y_m)
\]
for each $X,Y_1,\dots,Y_m \in S^*$
and
\[
\pi_i^{X_1,\dots,X_n} : 1 \to (X_1\timesdots X_n,X_i) \quad (i=1,\dots,l)
\]
for each $X_1,\dots,X_n \in S^*$
and $E^{\cattheories_\varsorts}_\mathrm{pair}$ consists of
\[
\pi_i \circ \ab<x_1,\dots,x_n> \approx x_i,\quad \ab<x_1,\dots,x_n> \circ y \approx \ab<x_1\circ y,\dots,x_n\circ y>,\quad \ab<\pi_1,\dots,\pi_n> \approx \Id.
\]

Again, any algebra $\ab\{A_{(X, Y)}\}$ of $(\Sigma^{\cattheories_\varsorts},E^{\cattheories_\varsorts})$ can be seen as a category $\bfA$ and it has morphisms $\left\llbracket \pi_i^{X_1,\dots,X_n}\right\rrbracket : X_1\timesdots X_n \to X_i$ and $\left\llbracket \ab<-,\dots,->_{W,X_1,\dots,X_n}\right\rrbracket(f_1,\dots,f_n) : W \to X_1\timesdots X_n$ for any $f_i : W \to X_i$ ($i=1,\dots,n$) in $\bfA$.
We just write $\pi_i^{X_1,\dots,X_n}$ or $\pi_i$ for $\left\llbracket \pi_i^{X_1,\dots,X_n}\right\rrbracket$ and $\ab<f_1,\dots,f_n>$ for $\left\llbracket \ab<-,\dots,->_{W,X_1,\dots,X_n}\right\rrbracket(f_1,\dots,f_n)$. We show that these morphisms make products.

Let $f_i : {W} \to {X}_i$ ($i=1,\dots,n$) be morphisms in $\bfA$.
If $h : {W} \to {X}_1 \timesdots {X}_n$ in $\bfA$ satisfies $\left\llbracket \pi_i^{X_1,\dots,X_n}\right\rrbracket \circ h = f_i$ for any $i=1,\dots,n$, then $h = \ab<f_1,\dots,f_n>$ since
\[
h = \ab<\pi_1,\dots,\pi_n>\circ h = \ab<\pi_1\circ h,\dots,\pi_n\circ h> = \ab<f_1,\dots,f_n>
\]
where each equality holds by the equations in $E^{\cattheories_\varsorts}$.

Also, it is not difficult to see that a morphism of $\Alg(\varsign^{\cattheories_\varsorts}, E^{\cattheories_\varsorts})$ corresponds to a product-preserving functor between $\varsorts$-sorted Lawvere theories.
\end{proof}

\subsection{$\varsorts$-sorted cartesian closed categories}

Let $\varsorts$ be a set.
We introduce $\varsorts$-\emph{sorted} cartesian closed categories ($\varsorts$-sorted CCCs).
\begin{definition}
We define a set $\Ty_\varsorts$ inductively as follows:
\begin{itemize}
	\item $X \in \Ty_\varsorts$ for any $X \in \varsorts$,
	\item $1 \in \Ty_\varsorts$,
	\item if $X,Y \in \Ty_\varsorts$, then $X\times Y \in \Ty_\varsorts$,
	\item if $X,Y \in \Ty_\varsorts$, then $Y^X \in \Ty_\varsorts$,
\end{itemize}
where $1$, $X\times Y$, $Y^X$ are formal expressions.
We call $\Ty_\varsorts$ the \emph{free pointed bi-magma} generated by $\varsorts$.
\end{definition}

\begin{definition}
Let $\varsign$ be a functor from the discrete category $\Ty_\varsorts \times \Ty_\varsorts$ to $\catset$, i.e., $\varsign$ is a family of sets indexed by pairs of two elements in $\Ty_\varsorts$.
We define the category $F_{\catccc_\varsorts}\varsign$ as follows.
$F_{\catccc_\varsorts}\varsign$ has an object $\varobj$ for each $\varobj\in \Ty_\varsorts$ and morphisms
\begin{itemize}
	\item $f : \varobj \to \varobj[1]$ for each $f \in \varsign(\varobj,\varobj[1])$,
	\item $\Id_{\varobj} : \varobj \to \varobj$ for each object $\varobj$,
	\item $!_{\varobj} : \varobj \to 1$ for each object $\varobj$,
	\item $\projection_i : \varobj_1 \times \varobj_2 \to \varobj_i$ for each $i=1,2$ and objects $\varobj_1,\varobj_2$, 
	\item $\ev{\varobj}{\varobj[1]} : \varobj[1]^{\varobj}\times \varobj \to \varobj[1] $ for each pair of objects $\varobj,\varobj[1]$,
	\item $g \circ f : \varobj \to \varobj[2]$ if $f : \varobj \to \varobj[1] $ and $g : \varobj[1] \to \varobj[2]$ are morphisms,
	\item $\ab< f_1,f_2> : \varobj \to \varobj_1 \times \varobj_2$ if $f_i : \varobj \to \varobj_i$ ($i=1,2$) are morphisms,
	\item $\currying f : \varobj \to \varobj[2]^{\varobj[1]}$ if $f : \varobj \times \varobj[1] \to \varobj[2]$ is a morphism,
\end{itemize}
and $\Id_{\varobj}$ and $g \circ f$ satisfy the laws of identity and composition, $\projection_i$ and $\ab<f_1,f_2>$ satisfy, for any $f_i : \varobj \to \varobj_i$, $g : W \to X$,
\begin{align*}
\pi_i\circ \ab<f_1,f_2> = f_i~ (i=1,2),\quad \ab<\pi_1,\pi_2> = \Id_{\varobj_1\times \varobj_2},\quad \ab<f_1,f_2>\circ g = \ab<f_1\circ g,f_2\circ g>,\quad
!_X \circ g = {!_W},
\end{align*}
and $\ev{\varobj}{\varobj[1]}$ and $\lambda f$ satisfy, for any $f : \varobj \to \varobj[2]^{\varobj[1]}$, $g : \varobj \to \varobj[2]^{\varobj[1]}$,
\[
\ev{\varobj[1]}{\varobj[2]}\circ (\lambda f\times \Id_{\varobj[1]}) = f,
\quad \lambda(\ev{\varobj[1]}{\varobj[2]}\circ (g\times \Id_{\varobj[1]})) = g
\]
where $f \times g$ is the shorthand for $\ab<f\circ \pi_1,g\circ \pi_2>$.
\end{definition}
We write $\catcfam_\varsorts$ for $F_{\catccc_\varsorts}\emptyset$ where $\emptyset$ is considered as the functor $\Ty_\varsorts \times \Ty_\varsorts \to \catset$ that maps every $(X,Y) \in \Ty_\varsorts \times \Ty_\varsorts$ into the empty set.

\begin{lemma}
$F_{\catccc_\varsorts}\varsign$ is a CCC.
\end{lemma}
\begin{proof}
We can show that $F_{\catccc_\varsorts}\Sigma$ is a Lawvere theory in a similar way to the proof of Proposition \ref{prop:lawvere-alg}.
Also, we show that $\lambda f$ is the unique morphism satisfying $\ev{\varobj[1]}{\varobj[2]}\circ (\lambda f\times \Id_{\varobj[1]}) = f$ for any $f : \varobj \times \varobj[1] \to \varobj[2]$.
Let $h : \varobj \to \varobj[2]^{\varobj[1]}$ be a morphism satisfying $\ev{\varobj[1]}{\varobj[2]}\circ (h\times \Id_{\varobj[1]}) = f$.
Then, $h = \lambda (\ev{\varobj[1]}{\varobj[2]}\circ (h\times \Id_{\varobj[1]})) = \lambda f$.
\end{proof}

\begin{definition}
An $\varsorts$-\emph{sorted CCC} is a CCC $\varccc$ together with a cartesian closed functor $\varsortedmor : \catcfam_\varsorts \to \varccc$ such that $\Ob\varccc = \Ty_\varsorts$ and $\varsortedmor$ is identity on objects.
We often call $\varccc$ an $\varsorts$-sorted CCC without mentioning $\varsortedmor$.
The full subcategory of $\catcfam_\varsorts\backslash\catccc$ consisting of $\varsorts$-sorted CCCs is denoted by $\catccc_\varsorts$.
\end{definition}

Note that $\iota : \catcfam_\varsorts \to \bfC$ is uniquely determined by its values $\iota(\pi_i)$, $\iota(\ev{Y}{Z})$.
Thus, we can think of an $\varsorts$-sorted CCC as a CCC with object set $\Ty_\varsorts$ where projections and evaluation maps are specified.

Like Proposition \ref{prop:sortedtheory}, any CCC $\bfC$ is equivalent to an $\varsorts$-sorted CCC for some set $\varsorts$ as follows.
\begin{proposition}\label{prop:sortedccc}
For any CCC $\varccc$, there exists a set $\varsorts$ and an $\varsorts$-sorted CCC $\catcfam_\varsorts \to \tilde\varccc$ such that $\varccc$ is equivalent to $\tilde\varccc$.
\end{proposition}
\begin{proof}
Construct $\tilde\bfC$ as follows.
Let $\Ob\ab(\tilde\varccc) = \Ty_{\Ob(\varccc)}$ and $\Hom_{\tilde\varccc}(X,Y) = \Hom_\varccc(\tilde X,\tilde Y)$ where $\tilde X$ is defined as (i) if $X \in \varsorts$, then $\tilde X = X$ (ii) if $X = 1$, then $\tilde X = 1$, (ii) if $X = X_1\times X_2$, then $\tilde X = \tilde X_1 \times \tilde X_2$, and (iv) if $X = X_2^{X_1}$, then $\tilde X = \tilde X_2^{\tilde X_1}$.
Then, we have a full and faithful functor $\tilde\varccc \to \varccc$ that is surjective on objects.
\end{proof}

In a similar way to the case for $\varsorts$-sorted Lawvere theories, we have
\begin{proposition}
$\catccc_\varsorts$ is algebraic.
\end{proposition}
\begin{proof}
Take the set of sorts as $\Ty_\varsorts \times \Ty_\varsorts$.
Let $(\Sigma^{\cattheories_\varsorts},E^{\cattheories_\varsorts})$ be the equational presentation for $\cattheories_\varsorts$ constructed in the proof of Proposition \ref{prop:lawvere-alg}.
We construct $\Sigma^{\catccc_\varsorts}$ by adding operations $\lambda^{X,Y,Z} : (X\times Y, Z) \to (X, Z^Y)$ and $\ev{X}{Y} : 1 \to (Y^X\times X, Y)$ for each $X,Y,Z \in \Ty_\varsorts$ to $\Sigma^{\cattheories_\varsorts}$, and construct $E^{\catccc_\varsorts}$ by adding equations
\[
\ev{}{}\circ (\lambda(g)\times \Id) = g
\quad \text{and} \quad
\lambda(\ev{}{}\circ (h \times \Id)) = h
\]
to $E^{\cattheories_\varsorts}$ where $f_1 \times f_2$ is the shorthand for $\ab<f_1 \circ \pi_1, f_2\circ \pi_2>$.
Then, we can check the equivalence $\Alg(\Sigma^{\catccc_\varsorts},E^{\catccc_\varsorts}) \simeq \catccc_\varsorts.$
\end{proof}

Also, we have an adjunction between $\catccc_\varsorts$ and $\catset^{\Ty_\varsorts \times \Ty_\varsorts}$.
Let $U_{\catccc_\varsorts} : \catccc_\varsorts \to \catset^{\Ty_\varsorts \times \Ty_\varsorts}$ be the functor that maps $\bfC$ to $(X,Y)\mapsto \Hom_\bfC(X,Y)$.
Then, we have the following:
\begin{proposition}
The functor $F_{\catccc_\varsorts} : \catset^{\Ty_\varsorts \times \Ty_\varsorts} \to \catccc_\varsorts$ is a left adjoint of $U_{\catccc_\varsorts}$.
\end{proposition}

So, we can call $F_{\catccc_\varsorts}\varsign$ the \emph{free $\varsorts$-sorted CCC} generated by $\Sigma$.

\section{Natural systems}

The first two subsections of this section are reviews of \cite{baues85}, \cite{jp05} and \cite{jp06}.

\subsection{Natural Systems and Linear Extensions}

Let $\bfC$ be a category.

\begin{definition}
The category $\catfact\bfC$ of \emph{factorizations} in $\bfC$ is defined as follows.
Objects of $\catfact\bfC$ are morphisms in $\bfC$, and morphisms $(a,b): f \to g$ in $\catfact\bfC$ for $f: A \to B$, $g: A'\to B'$ are commutative diagrams
\[
\begin{tikzcd}
	A \ar[d,"f"] & A' \ar[l,"a"] \ar[d, "g"]\\
	B \ar[r, "b"] & B'
\end{tikzcd}
\]
in $\bfC$.
\end{definition}

\begin{definition}
A \emph{natural system} on $\bfC$ is a functor $F: \catfact\bfC \to \catab$.
A \emph{morphism of natural systems} is a natural transformation.
We write $\catnat_\bfC$ for the category of natural systems on $\bfC$.
That is, $\catnat_\bfC = \catab^{\catfact\bfC}$.
\end{definition}

Notation: We write $D_f$ for $D(f)$.
For $a: C \to D$, $f: A \to C$, $g: D \to B$,
we write $a_*$ for $D(1_A,a): D_f \to D_{a\circ f}$ and $a^*$ for $D(a,1_B): D_g \to D_{g\circ a}$.

\begin{definition}
Let $D$ be a natural system on $\bfC$.
A \emph{linear extension} of $\bfC$ by $D$, written $D \to \bfE \xrightarrow{p} \bfC$, is a category $\bfE$ together with a full functor $p: \bfE \to \bfC$ that is identity on objects and, for each $f: X \to Y$ in $\bfC$, a transitive and effective action of $D_f$ on $p^{-1}(f)\subset \Hom_\bfE(X,Y)$
\[
{+} : D_f\times p^{-1}(f) \to p^{-1}(f)
\]
such that
\[
(\xi + \tilde f)\circ (\eta + \tilde g) = f_* \eta + g^* \xi + \tilde f \circ \tilde g
\]
where $f: X \to Y$, $g: W \to X$ are in $\bfC$, $\tilde f \in p^{-1}(f)$, $\tilde g \in p^{-1}(g)$, and $\xi \in D_f$, $\eta \in D_g$.

Two extensions $D \to \bfE\xrightarrow{p}\bfC$ and $D \to \bfE'\xrightarrow{p'}\bfC$ are \emph{equivalent} if there is an isomorphism of categories $\epsilon: \bfE \to \bfE'$ satisfying $p'\circ \epsilon = p$ and $\epsilon(\xi+\tilde f)=\xi + \epsilon(\tilde f)$.
\end{definition}

\begin{definition}
Let $D$ be a natural system on $\bfC$.
A \emph{trivial linear extension} $D\rtimes \bfC$ is defined as
\[
\Hom_{D\rtimes \bfC}(X,Y) = \coprod_{f\in \Hom_\bfC(X,Y)} D_f 
\]
and composition given by
\[
\xi \circ \eta = f_* \eta + g^* \xi
\]
for $f: X \to Y$, $g: W \to X$ in $\bfC$ and $\xi \in D_f$, $\eta\in D_g$.
The group action $D_f \times D_f \to D_f$ is the addition in $D_f$.
\end{definition}

\begin{definition}
We say that a linear extension $D \to \bfE \xrightarrow{p} \bfC$ \emph{splits} if there is a functor $s : \bfC \to \bfE$, called a \emph{section}, satisfying $ps = \Id_\bfC$.
\end{definition}
It is not difficult to see that a linear extension splits if and only if it is equivalent to the trivial linear extension.

\begin{proposition}\label{prop:nat-beck}\cite{jp05}
Let $O$ be the set of objects of $\bfC$.
There is an equivalence of categories $\catnat_\bfC \simeq (\catcat_O/\bfC)_\abobjs$.
\end{proposition}

\subsection{Lawvere theories and cartesian natural systems}
\begin{definition}
Let $\varlaw$ be a small category with finite products and $D$ be a natural system on $\varlaw$.
We say that $D$ is \emph{cartesian} if, for any $f: X \to X_1\times\dots \times X_n$ and projections $\pi_k: X_1\times\dots\times X_n \to X_k$, the group homomorphism
\begin{align}\label{eqn:cartesian}
\begin{split}
	D_f &\to D_{\pi_1\circ f}\times \dots \times D_{\pi_n\circ f}\\
	\xi &\mapsto (\pi_{1*}\xi,\dots,\pi_{n*} \xi)
\end{split}
\end{align}
is an isomorphism.
\end{definition}

\begin{lemma}\cite{jp06}\label{lem:cartnat}
Let $\bfC$ be a category with finite products and $D \to \bfE \xrightarrow{p} \bfC$ be a linear extension of $\bfC$ by a natural system $D$.
Then, $D$ is cartesian if and only if $\bfE$ has finite products and $p$ is a product-preserving functor.
Also, in that case, if $\bfC$ has a structure of $\varsorts$-sorted Lawvere theory $\catfam_\varsorts \to \bfC$, then there exists $\catfam_\varsorts \to \bfE$ such that $p$ is a morphism between $\varsorts$-sorted Lawvere theories.
\end{lemma}
\begin{proof}
For each projection $\pi_k: X_1\times\dots \times X_n \to X_k$ in $\varlaw$, choose $\tilde\pi_k \in p^{-1}(\pi_k)$.
Then, $0_{\pi_k}\circ \tilde f = \pi_{k*}\tilde f$ for any $\tilde f \in p^{-1}(f)$ and we have the following commutative diagram:
\[
\begin{tikzcd}[column sep=17ex]
	\Hom_{\bfE}(X,X_1\times\dots\times X_n) \ar[r,"\tilde f \mapsto (\pi_{1*}\tilde f{,}\dots{,}\pi_{n*} \tilde f)"] \ar[d,"p"'] & \Hom_{\bfE}(X,X_1)\times\dots\times \Hom_\bfE(X,X_n) \ar[d,"p"]\\
	\Hom_\varlaw(X,X_1\times \dots \times X_n) \ar[r, "\sim"] & \Hom_\varlaw(X,X_1)\times\dots \times \Hom_\varlaw(X,X_n).
\end{tikzcd}
\]
Then, it is easy to see that $\bfE$ has and $p$ preserves finite products if and only if
(\ref{eqn:cartesian}) is bijective for each $f$.
Then, our first statement is followed by Lemma \ref{lem:equivariant}.

For the second statement, given $\iota : \catfam_\varsorts \to \bfC$, define $\tilde\iota : \catfam_\varsorts \to \bfE$ by $\tilde\iota(\pi_i) = \tilde\pi_i$.
\end{proof}

\begin{lemma}\label{lem:equivariant}
For any group homomorphism $f : G_1 \to G_2$ and $f$-equivariant map $x : X_1 \to X_2$ 	for sets $X_i$ with transitive and effective $G_i$-action, $x$ is bijective if and only if $f$ is an isomorphism.
\end{lemma}
For the proof, see \cite[Lemma 3.5]{jp91} for example.
\begin{theorem}\cite{jp06}
For any $\varsorts$-sorted Lawvere theory $\varlaw$, there is an equivalence of categories
\[
\mathbf{CNat}_\varlaw\simeq (\cattheories_\varsorts/\varlaw)_{\mathrm{ab}}.
\]
\end{theorem}

\subsection{Cartesian Closed Natural Systems}
In this subsection, we introduce a notion of \emph{cartesian closed natural systems} and show that the category of cartesian closed natural systems on $\bfC$ is equivalent to $(\catccc_\varsorts/\bfC)_{\mathrm{ab}}$.

Let $D$ be a cartesian natural system on a cartesian closed category $\bfC$.
We write $\ev{Y}{Z}$ for the evaluation map $Z^Y \times Y \to Z$.

Note that for $g: X \to W$ in $\bfC$ and an object $Y$, from the maps
\[
X \times Y \xrightarrow{\pi_1} X \xrightarrow{g} W
\]
we obtain a group homomorphism
\begin{align*}
D_g &\to D_{g\circ \pi_1}\\
\xi &\mapsto \pi_{1}^*\xi
\end{align*}
and, we write $\phi_g$ for the composed map 
\[
D_g \times D_{\pi_2} \xrightarrow{\pi_1^*\times 1_{D_{\pi_2}}} D_{g\circ \pi_1} \times D_{\pi_2} \xrightarrow{\sim} D_{g\times 1_Y}
\]
where $\pi_2 : X \times Y \to Y$.
Note that, for $f : X \times Y \to Z$, $\xi \in D_{\lambda f}$ and $\upsilon \in D_{\pi_2}$, $\ev{Y}{Z*}\phi_{\lambda f}(\xi,\upsilon)$ is an element of $D_{\ev{Y}{Z}\circ (\lambda f \times 1_Y)} = D_f$.
\begin{definition}
We say that a cartesian natural system $D$ is cartesian \emph{closed} if, for any $f: X\times Y \to Z$, the group homomorphism
\begin{align}\label{eqn:cartesianclosed}
\begin{split}
	D_{\lambda f} &\to D_{f}\\
	\xi &\mapsto \ev{Y}{Z*} \phi_f(\xi,0)
\end{split}
\end{align}
is an isomorphism.
\end{definition}

\begin{lemma}\label{lem:ccnat}
Let $\bfE \xrightarrow{p} \bfC$ is a linear extension by $D$.
Then $D$ is cartesian closed if and only if $\bfE$ is cartesian closed and $p$ is a cartesian closed functor.
Also, in that case, if a structure of an $\iota : \varsorts$-sorted CCC $\catcfam_\varsorts \to \bfC$ is given, then there exists $\tilde\iota : \catcfam_\varsorts \to \bfE$ such that $p$ is a morphism between $\varsorts$-sorted CCCs.
\end{lemma}
\begin{proof}
By Lemma \ref{lem:cartnat}, it suffices to check that, for a cartesian $D$, $D$ is cartesian closed if and only if $\bfE$ has and $p$ preserves exponentials.
For each evaluation map $\ev{Y}{Z}: Z^Y \times Y \to Z$ in $\bfC$, choose an arbitrary morphism $\tilde{\mathrm{ev}}^{Y}_{Z}$ in $p^{-1}(\ev{Y}{Z})$.
Then, we have
\[
\begin{tikzcd}[column sep=17ex]
	\Hom_\bfE(X,Z^Y) \ar[r,"\tilde f \mapsto \tilde{\mathrm{ev}}_Z^Y \circ (\tilde f \times 1_Y)"] \ar[d,"p"] & \Hom_\bfE(X\times Y,Z) \ar[d,"p"]\\
	\Hom_\bfC(X,Z^Y) \ar[r,"\sim"] & \Hom_\bfC(X\times Y,Z).
\end{tikzcd}
\]
$\bfE$ has and $p$ preserves exponentials if and only if the map
\begin{align*}
	p^{-1}(f) &\to p^{-1}(\ev{Y}{Z}\circ (f\times 1_Y))\\
	\tilde f &\mapsto \tilde{\mathrm{ev}}^{Y}_{Z} \circ (\tilde f \times 1_Y)
\end{align*}
is bijective for each $f$.
This map is equivariant with respect to the group homomorphism (\ref{eqn:cartesianclosed}).
Then, our first statement is followed by Lemma \ref{lem:equivariant}.

For the second statement, 
let $\tilde\iota(\ev{Y}{Z}) = \tilde{\mathrm{ev}}^Y_Z$ for any evaluation map $\ev{Y}{Z}$ in $\catcfam_\varsorts$.
\end{proof}

\begin{theorem}
For any $\varsorts$-sorted CCC $\iota : \catcfam_\varsorts \to \bfC$, we have an equivalence of categories
\[
\Xi : \catccnat_\bfC \xrightarrow{\sim} (\catccc_\varsorts/\bfC)_{\mathrm{ab}}.
\]
\end{theorem}
\begin{proof}
Let $D$ be a cartesian closed natural system.
The trivial linear extension $D \rtimes \bfC$ is cartesian closed by Lemma \ref{lem:ccnat} and define $\iota : \catcfam_\varsorts \to \bfC$ by $\iota(\Id)= 0$, $\iota(\pi_i) = 0$, $\iota(\ev{Y}{Z}) = 0$ where 0 is the zero in $D_f$ for $f = \iota(\Id),\iota(\pi_i),\iota(\ev{Y}{Z})$.
The addition $ D\rtimes \bfC \times_\bfC D \rtimes \bfC \to D \rtimes \bfC$ is defined as the addition in $D_f$s, and we can check that this provides an abelian group structure on $D \rtimes \bfC$.

Conversely, let $\bfE \xrightarrow{p} \bfC$ be an internal abelian group in $\catccc_\varsorts/\bfC$.
Define a natural system $D(p)$ on $\bfC$ as $D(p)_f = p^{-1}(f)$ and $D(p)(a,b)(\tilde f) = 0_a \tilde f 0_b$ where $0_g$ is the zero in $D(p)_g = p^{-1}(g)$.
$D(p)$ is cartesian since
\[
D(p)_{\pi_1\circ f}\times \dots \times D(p)_{\pi_n \circ f} \ni \ab(\tilde f_1,\dots,\tilde f_n) \mapsto \ab< \tilde f_1,\dots, \tilde f_n > \in D(p)_f
\]
is the inverse of (\ref{eqn:cartesian}).
Also, $D(p)$ is cartesian closed since
\[
D(p)_{\mathrm{ev}_{Y,Z}\circ (f\times 1_Y)} \ni \tilde g \mapsto \lambda \tilde g \in D(p)_f
\]
is the inverse of (\ref{eqn:cartesianclosed}).
\end{proof}

\begin{definition}
Let $D$ be a cartesian closed natural system on an $\varsorts$-sorted CCC $\bfC$.
We define $\Der^{\catccc_\varsorts}(\bfC; D)$ as the abelian group of all morphisms $s : \bfC \to D \rtimes \bfC$ in $\catccc_\varsorts$ such that $ps = \Id_\bfC$ where $p : D \rtimes \bfC \to \bfC$ is the canonical projection.
\end{definition}

\begin{lemma}\label{lem:der-desc}
For a cartesian closed natural system $D$ on an $\varsorts$-sorted CCC $\iota : \catcfam_\varsorts \to \bfC$, there is an isomorphism
\[
\Der^{\catccc_\varsorts}(\bfC; D) \cong \ab\{d \in \prod_{f \in \Mor(\bfC)}D_f \middle\mid d(f\circ g) = f_*d(g) + g^*d(f),~d(\iota(\pi_i)) = d(\iota(\ev{Y}{Z}))=0\}.
\]
\end{lemma}
\begin{proof}
Any $s \in \Der^{\catccc_\varsorts}(\bfC;D)$ can written as $\Hom_\bfC(X,Y) \ni f \mapsto (df,f) \in \Hom_{D\rtimes \bfC}(X,Y)$, and since $s$ preserves compositions, the first equation is obtained.
The second and the last equations are derived from $s(\iota(\pi_i)) = \tilde\iota(\pi_i)=0$ and $s(\iota(\ev{Y}{Z}))=\tilde\iota(\ev{Y}{Z})=0$.
\end{proof}

\section{Baues--Wirsching cohomology}

Let $\bfC$ be a small category and $D$ be a natural system on $\bfC$.

For $n > 0$, we define $C^n_\BW(\bfC;D)$ as the abelian group of all functions
\[
f: N_n(\bfC) \to \bigcup_{g \in \Mor(\bfC)}D_g
\]
such that $f(\lambda_1,\dots,\lambda_n) \in D_{\lambda_1\dots\lambda_n}$.
Here, $N(\bfC)$ is the nerve of $\bfC$.

For $n=0$, let $C^0_\BW(\bfC;D)$ be the abelian group of all functions
\[
f: N_0(\bfC) = \Ob(\bfC) \to \bigcup_{A \in \Ob(\bfC)}D_A
\]
such that $f(A) \in D_A$ where $D_A = D_{1_A}$.
The addition in $C^n_\BW$ is given by pointwise addition in $D_f$s.

Define the coboundary map $\delta : C^{n}_\BW \to C^{n+1}_\BW$ as, for $n=0$,
\[
(\delta f)(\lambda) = \lambda_*f(A) - \lambda^*f(B)\quad (\lambda: A \to B \in N_0(\bfC))
\]
and for $n > 1$,
\[
(\delta f)(\lambda_1,\dots,\lambda_n) = \lambda_{1*}f(\lambda_2,\dots,\lambda_n)
+ \sum_{i=1}^{n-1} (-1)^if(\lambda_1,\dots,\lambda_i\lambda_{i+1},\dots,\lambda_n)
+ (-1)^n \lambda^*_n f(\lambda_1,\dots,\lambda_{n-1}).
\]
We can check $\delta\delta = 0$, so $(C^n_\BW(\bfC;D),\delta)$ forms a cochain complex.
\begin{definition}\cite{baues85}
The $n$-th \emph{Baues--Wirsching cohomology group} $H^n_\BW(\bfC;D)$ is the $n$-th cohomology group of the cochain complex $(C^\bullet_\BW(\bfC;D),\delta)$.
\end{definition}
It is known that the Baues--Wirsching cohomology is invariant under equivalences of categories in the following sense.
\begin{proposition}\cite{baues85}\label{prop:bw-invariant}
For any two small categories $\bfC$, $\bfC'$ with equivalence $\phi : \bfC \to \bfC'$ and a natural system on $\bfC'$, $\phi$ induces an isomorphism
\[
H^n_\BW(\bfC; \phi^*D) \cong H^n_\BW(\bfC'; D)
\]
for any $n \ge 0$ where $\phi^*D$ is the natural system given by $\phi^*D_f = D_{\phi(f)}$, $a_* = \phi(a)_*$, $b^* = \phi(b)^*$ for $f,a,b \in \Mor(\bfC)$.
\end{proposition}

It is known that $H^2_\BW(\bfC;D)$ classifies linear extensions of $\bfC$ by $D$.
\begin{proposition}\cite{baues85}
Let $M(\mathbf C; D)$ be the set of equivalence classes of linear extensions of $\bfC$ by $D$.
There is a natural bijection $M(\bfC; D) \cong H^2_\BW(\mathbf C; D)$ that maps the trivial linear extension to the zero in $H^2_\BW(\mathbf C; D)$.
\end{proposition}

The Baues--Wirsching cohomology of $\bfC$ can be written as an Ext over $\catnat_\bfC = \catab^{\catfact{\bfC}}$.
Let $\cZ_\bfC$ be the natural system on $\bfC$ such that for each morphism $f : X \to Y$ in $\bfC$, $(\cZ_\bfC)_f$ is the free abelian group generated by $f$, and for each $a : X' \to X$, $b : Y \to Y$, $a^* : (\cZ_\bfC)_f \to (\cZ_\bfC)_{fa}$ and $b_* : (\cZ_\bfC)_f \to (\cZ_\bfC)_{bf}$ are the isomorphisms sending the generator $f$ to the generators $a^* f$ and $b_*f$, respectively.

\begin{proposition}\cite{baues85}\label{prop:bwext}
For any natural system $D$ on $\bfC$, there is an isomorphism
$H^n_\BW(\bfC;D) \cong \Ext_{\catnat_\bfC}^n(\cZ_\bfC,D)$.
\end{proposition}

It is proved that, for any $n\ge 2$, $H^n_\BW(\bfC; D) = 0$ for a free $\bfC$ in $\catcat_O$ and a natural system $D$ on $\bfC$ \cite{baues85}, and also for a free $\bfC$ in $\cattheories_\varsorts$ and a cartesian natural system $D$ on $\bfC$ \cite{jp06}.
We show that the same holds for a free $\bfC$ in $\catccc_\varsorts$ and a cartesian closed natural system $D$ on $\bfC$.

\begin{proposition}\label{prop:freeccc}
For any free $S$-sorted CCC $\bfC$ and a cartesian closed natural system $D$ on $\bfC$, we have $H^n_\BW(\bfC; D) = 0$ for $n \ge 2$.
\end{proposition}
\begin{proof}
For any linear extension $D \to \bfE \xrightarrow{p} \bfC$, by Lemma \ref{lem:ccnat}, we can equip $\bfE$ with a structure of an $S$-sorted CCC.
Since $\bfC$ is a free $S$-sorted CCC, we can construct a morphism $s : \bfC \to \mathbf{E}$ in $\catccc_S$ such that $ps = \Id_\bfC$.
Therefore the extension splits, and by Proposition \ref{prop:bwext}, we have $H^n_\BW(\bfC;D) = 0$ for any $n \ge 2$.
\end{proof}

\section{Equivalence}

In \cite[Theorem 4, page 4.2]{quillen2006homotopical}, Quillen showed that any algebraic category $\cC$ has a simplicial model structure.
One of the important fact we use here is that, for any $X \in \Ob(\cC)$, we can take a cofibrant replacement $Y_\bullet \to X$ and $Y_\bullet$ is degreewise free.
We call such $Y_\bullet \to X$ a simplicial free resolution of $X$.
In this paper we take $\cC = \catccc_\varsorts$.

Note that if $p : \bfC' \to \bfC$ is an $\varsorts$-sorted CCC over $\bfC$, then $p$ induces a morphism $\catfact{\bfC'} \to \catfact{\bfC}$, so any natural system $D$ on $\bfC$ can be considered as a natural system on $\bfC'$.

\begin{definition}
Let $\bfC$ be an $\varsorts$-sorted CCC and $D$ be a cartesian closed natural system on $\bfC$.
Then the $n$-th \emph{Quillen cohomology group} of $\bfC$ with coefficients in $D$, written $HQ^n_{\catccc_\varsorts}(\bfC;D)$, is given as
\[
\HQ^n_{\catccc_\varsorts}(\bfC;D) = H^n(\Der^{\catccc_\varsorts}(\bfF_\bullet; D))
\]
where $\bfF_\bullet$ is a simplicial free resolution of $\bfC$ in $\catccc_\varsorts$.
\end{definition}

The goal of this paper is to show
\[
\HQ^n_{\catccc_\varsorts}(\bfC;D)
\cong H^{n+1}_\BW(\bfC;D)
\]
for any $n \ge 1$.

Let $\bfC$ be an $\varsorts$-sorted CCC and $D$ be a cartesian closed natural system on $\bfC$.
Let $p : \bfF \to \bfC$ be an $\varsorts$-sorted CCC over $\bfC$ and suppose that $\bfF$ is freely generated by $\{f_i\}_{i\in I}$.

We define $\tilde C^0_\BW(\bfF; D)$ as the subgroup of $\ker(\delta : C^1_\BW(\bfF;D) \to C^2_\BW(\bfF;D))$ consisting of $\phi$ such that $\phi(f_i) = 0$ for $i\in I$.
Note that any $\phi \in \ker(\delta : C^1_\BW(\bfF;D) \to C^2_\BW(\bfF;D))$ satisfies $\phi(a\circ b) = a_* \phi(b) + b^* \phi(a)$, so $\phi(a\circ b)$ is determined by $\phi(a)$ and $\phi(b)$.

For any $g_i : X \to X_i$ ($i=1,2$) in $\bfF$, since $\phi(g_i) = \phi(\pi_i \circ \ab<g_1,g_2>) = \pi_{i*}\phi(\ab<g_1,g_2>) + \ab<g_1,g_2>^*\phi(\pi_i)$ and $\xi \mapsto (\pi_{1*}\xi,\pi_{2*}\xi)$ is an isomorphism,
$\phi(\ab<g_1,g_2>)$ is determined by $\phi(g_i)$, $\phi(\pi_i)$ for $i=1,2$.
Similarly, for any $g : X \times Y \to Z$, since $\phi(g) = \phi(\ev{Y}{Z}\circ (\lambda g \times \Id_Y)) = \ev{Y}{Z*}\phi(\lambda g\times \Id_Y) + (\lambda g \times \Id_Y)^* \phi(\ev{Y}{Z})$,
$\phi(\lambda g)$ is determined by $\phi(g)$, $\phi(\ev{Y}{Z})$.
So, $\phi \in \ker(\delta : C^1_\BW(\bfF;D) \to C^2_\BW(\bfF;D))$ is uniquely determined by $\phi(f_i)$, $\phi(\pi_i)$, $\phi(\ev{Y}{Z})$, and $\phi \in \tilde C_\BW^0(\bfF;D)$ is uniquely determined by $\phi(\pi_i)$, $\phi(\ev{Y}{Z})$.
(In particular, $\tilde C^0_\BW(\bfF;D)$ is the same group for any free $\varsorts$-sorted CCC $\bfF$.)
Also, by Lemma \ref{lem:der-desc}, $\phi \in \Der^{\catccc_\varsorts}(\bfF;D)$ is determinded by $\phi(f_i)$ ($i\in I$), and from these observation, we get the following.

\begin{lemma}
$\tilde C_\BW^0(\bfF; D)\oplus \Der^{\catccc_\varsorts}(\bfF; D)\cong \ker(\delta : C^1_\BW(\bfF;D) \to C^2_\BW(\bfF;D))$.
\end{lemma}

For $n > 0$, let $\tilde C^n_\BW(\bfF; D) = C^n_\BW(\bfF;D)$. Then $\tilde C^\bullet_\BW(\bfF; D)$ forms a cochain complex.
Then, we get $H^1\ab(\tilde C_\BW^\bullet(\bfF; D)) \cong \Der^{\catccc_\varsorts}(\bfF;D)$ and $H^n\ab(\tilde C_\BW^\bullet(\bfF;D)) =0 $ for $n \neq 1$.

\begin{remark}
Even though the proof of $\HQ^n_{\catccc_\varsorts} \cong H^{n+1}_\BW$ we are going to give is quite similar to that for Lawvere theories in \cite{jp06}, they claimed and used an incorrect proposition $\Der^{\cattheories_\varsorts}(\bfF;D) \cong \ker(\delta : C^1_\BW(\bfF;D) \to C^0_\BW(\bfF; D))$, which makes their proof invalid.
Our discussion above corrects their mistake.
\end{remark}

\begin{theorem}
For $n > 0$, an $\varsorts$-sorted cartesian closed category $\bfC$, and a cartesian closed natural system $D$ on $\bfC$,
\[
\HQ^n_{\catccc_\varsorts}(\bfC;D) \cong H^{n+1}_\BW(\bfC;D).
\]
\end{theorem}
\begin{proof}
Let $\epsilon : \bfF_\bullet \to \bfC$ be a simplicial free resolution in $\catccc_I/\bfC$.
For each two objects $X,Y$, the simplicial object $\bfF_\bullet$ induces the simplicial set $\Hom_{\bfF_\bullet}(X,Y)$ given by
$\Hom_{\bfF_\bullet}(X,Y)_k = \Hom_{\bfF_k}(X,Y)$,
and $\epsilon$ induces a weak equivalence from $\Hom_{\bfF_\bullet}(X,Y)$ to $\Hom_\bfT(X,Y)$ considered as a constant simplicial set.

Consider the double complex $\tilde C^{\bullet}_\BW(\bfF_\bullet; D)$ and two spectral sequences $\fstspec^{pq}$, $\sndspec^{pq}$ converging to the cohomology of the total complex.
For $\fstspec^{pq}_1$, we have
\[
\fstspec_1^{pq} = H^q \ab(\tilde C^{\bullet}_\BW(\bfF_p; D)) =
\begin{cases}
\Der^{\catccc_\varsorts}(\bfF_p;D), & q = 1,\\
0, & q \neq 1,
\end{cases}
\]
so, $\fstspec_2^{pq} \Rightarrow \HQ^{p+q-1}(\bfC; D)$.

For $\sndspec_1^{pq}$, we have
\[
\sndspec_1^{pq} = H^q\ab(\tilde C^{p}_\BW(\bfF_\bullet; D)).
\]
For $Y_0,\dots,Y_p \in\Ob(\bfC)=\Ob(\bfF_q)$, consider the simplicial set
\[
S_\bullet^{Y_0,\dots,Y_p} = \Hom_{\bfF_\bullet}(Y_1,Y_0)\times \dots \times \Hom_{\bfF_\bullet}(Y_{p},Y_{p-1}).
\]
By the definition of Baues--Wirsching cochain complexes, for any $p > 0$,
\[
C^{p}_{\BW}(\bfF_q; D) \cong \prod_{Y_0,\dots,Y_{p}} C^q\ab(S_\bullet^{Y_0,\dots,Y_p}; D_{(-)})
\]
where the right-hand side is the product of cochain complexes of simplicial sets $S_\bullet^{Y_0,\dots,Y_p}$ with coefficients in $D_{f_1\dots f_{p}}$ on the connected component of $S_\bullet^{Y_0,\dots,Y_p}$ corresponding to $(f_1,\dots,f_{p}) \in \Hom_\bfC(Y_1,Y_0)\times \dots \times \Hom_\bfC(Y_{p},Y_{p-1})$.

We have $H^n\ab(S_\bullet^{Y_0,\dots,Y_p}; D_{(-)}) = 0$ for $n > 0$ since we have a weak equivalence between $\bfF_\bullet$ and $\bfC$, a constant simplicial object.
Therefore, $\sndspec_1^{pq} = 0$ for $p,q > 0$.
For $p > 0$ and $q=0$, we have
\[
\sndspec_1^{p0} = \prod_{Y_0,\dots,Y_{p+1}} H^0\ab(\Hom_{\bfF_\bullet}(Y_1,Y_0)\times \dots \times \Hom_{\bfF_\bullet}(Y_{p},Y_{p-1}); D_{(-)}) = C_\BW^{p}(\bfC;D).
\]
For $p=0$, $\sndspec_1^{0q} = 0$ since $\tilde C^0_\BW(\bfF_i;D) \to \tilde C^0_\BW(\bfF_{i-1};D)$ is an isomorphism.
Thus $\sndspec_2^{pq} \Rightarrow H_\BW^{p}(\bfC; D)$.
\end{proof}

Note that if we have an $\varsorts$-sorted CCC $\bfC$, an $\varsorts'$-sorted CCC $\bfC'$, a cartesian closed natural system $D$ on $\bfC'$, and an equivalence $\phi : \bfC \xrightarrow{\sim} \bfC'$ of categories, then, by Proposition \ref{prop:bw-invariant} and the above theorem, we have
\[
\HQ^n_{\catccc_\varsorts}(\bfC; \phi^*D) \cong \HQ^n_{\catccc_{\varsorts'}}(\bfC'; D).
\]
In other words, $\HQ^n(\bfC;D)$ does not depend on the choice of sortings of $\bfC$.

\section{Open problem}

In \cite{jp06}, Jibladze and Pirashvili showed that the Quillen and Baues--Wirsching cohomologies of an $\varsorts$-sorted Lawvere theory $\bfC$ with coefficients in a cartesian natural system $D$ is also isomorphic to
\(
\Ext^n_{\catcnat_\bfC}(\Omega_\bfC^{\cattheories_\varsorts},D).
\)
Does this extend to $\varsorts$-sorted CCCs?
That is, is there an isomorphism
\[
\HQ^n_{\catccc_\varsorts}(\bfC;D) \cong \Ext^n_{\catccnat_\bfC}\ab(\Omega_\bfC^{\catccc_\varsorts},D)
\]
for any $\varsorts$-sorted CCC $\bfC$ and a cartesian closed natural system $D$ on $\bfC$?

Here, since Quillen showed in \cite{quillen1970co} that there is a spectral sequence
\[
E^{pq}_2 = \Ext^p_{(\cC/X)_\abobjs}\ab(\HQ_q^{\cC}(X),M) \Rightarrow \HQ^{p+q}_{\cC}(X;M)
\]
for any object $X$ of an algebraic category $\cC$ and a Beck module $M$ over $X$, it is the same to ask whether $\HQ^{\catccc_\varsorts}_q(\bfC) = 0$ for any $q > 0$.

\bibliographystyle{elsarticle-num}
\bibliography{jpaa}
\end{document}